\documentclass{RSMUP}

\address[rodolphe.richard@normalesup.org]{Rodolphe \textsc{RICHARD}}

\usepackage[frenchb,english]{babel}
\usepackage{mathrsfs}
\usepackage{pdfpages}
\newif\ifXELATEX
\XELATEXfalse
\ifXELATEX
	\usepackage{fontspec} 
	\defaultfontfeatures{Mapping=tex-text} 
	\usepackage{xunicode} 
	\usepackage{xltxtra} 
\else 
	\usepackage[T1]{fontenc} 
	\usepackage[utf8]{inputenc} 
\fi

\usepackage[colorlinks=true,citecolor=blue]{hyperref}

\received[\today]{\today}

\newtheorem{theorem}{Theorem}[section]

\theoremstyle{definition}

\newtheorem{example}[theorem]{Exemple}
\newtheorem{remark}[theorem]{Remarque}

\newtheorem{theo}{Théorème}
\newtheorem{coro}{Corollaire}
\newtheorem{prop}{Proposition}
\newtheorem{lemme}{Lemme}

\newcommand{\QQ}{{\mathbf{Q}}}
\newcommand{\eps}{\varepsilon}
\newcommand{\abs}[1]{{\left|{#1}\right|}}

\newcommand{\floor}[1]{{\left\lfloor{#1}\right\rfloor}}
\newcommand{\pgcd}[2]{\mathrm{pgcd}\left(#1,#2\right)}

\title[Des~$\pi$-exponentielles~I]{Des~$\pi$-exponentielles~I:\\
Vecteurs de Witt annulés par Frobénius\\ et Algorithme de (leur) rayon de convergence}

\author[Rodolphe Richard]{Rodolphe Richard}

\makeatletter
\@addtoreset{paragraph}{section}
\renewcommand{\theparagraph}{\textbf{\@arabic\c@section.\@arabic\c@paragraph}}
\makeatother

\begin{document}

\maketitle

\selectlanguage{french}
\begin{abstract}
 Nous commençons par un énoncé abstrait sur la structure des noyaux itérés de l’endomorphisme de Frobenius sur l’anneau des vecteurs de Witt à valeurs dans un anneau d’entiers d’une extension ultramétrique de~$\QQ_p$. C’est conséquence et reformulation de résultats de Pulita dans sa théorie des «~$\pi$-exponentielles~». Nous retraduisons ensuite cet énoncé en termes plus élémentaires, d’intégralité et de solubilité de séries exponentielles de polynômes. Nous montrons ensuite comment en déduire une formule explicite de leur rayon de convergence, et même de leur fonction rayon de convergence. On retrouve ainsi, en termes élémentaires, l’algorithme de Christol. Nous ajoutons ça et là quelques raffinements et observations.
\end{abstract}

\selectlanguage{english}
\begin{abstract} Our object is the theory of~``$\pi$-exponentials" Pulita developed in his thesis, generalising Dwork’s and Robba’s exponentials and extending Matsuda’s work:
\newline
We start with an abstract algebra statement about the structure of the kernel of iterations of the Frobenius endomorphism on the ring of Witt vectors \emph{with coordinates in the ring of integers of an ultrametric extension of~$\QQ_p$}. Provided sufficiently (ramified) roots of unity are available, it is, unexpectedly simply, a principal ideal with respect to an explicit generator essentially given by Pulita’s $\pi$-exponential. This result is a consequence and a reformulation of core facts of Pulita’s theory. It happened to be simpler to prove directly than reformulating Pulita's results.
\newline
Its translation in terms of series is very elementary, and gives a criterion for solvabilty and integrality for $p$-adic exponential series of polynomials. We explain how to deduce an explicit formula of their radius of convergence, and even the function radius of convergence. We recover this way, in elementary terms, with a new proof, and important simplifications, an algorithm of Christol based similarly on Pulita’s work. One concrete advantage is: one can easily prove rigorous complexity bounds about the implied algorithm from our explicit formula. We also add there and there refinements and observation, notably hinting some of the finer informations that can also given by the algorithm. 
\newline
One of the appendix produce a computation which gives finer estimates on the coefficients of these series. It should provide useful in proving complexity bounds for various computational use involving these series. It is not apparent yet in the present work, but should be in latter projected developments, the series under consideration are the base object for some exponential sums on finite fields via $p$-adic approach, namely via rigid cohomology with rank one coefficients.
\newline
Convergence radius and coefficients estimates are involved studying the efficiency of computational implementations of these objects. The understanding of convergence radius of $p$-adic differential equations is a subject undergoing active developments, and we here provide a fine theoretical and computational study of the simplest of cases.
\newline
This initiate a projected series of articles. We start here, with the case of the affine line as a base space, a Witt vectors paradigm. This provides an alternative purely algebraic approach of Pulita’s theory; the richness of Witt vectors theory allow suppleness and efficiency in working with~$\pi$-exponentials, which will prove efficient later in the series.
\end{abstract}
\selectlanguage{french}
\setcounter{secnumdepth}{5}
\setcounter{tocdepth}{1}



\tableofcontents
\setcounter{section}{-1}
\section{Introduction thématique et Contexte historique} Certaines questions de rayon de convergence de série entières à coefficients $p$-adiques apparaissent manifestement en étudiant l’exponentielle de Dwork~$\exp(\pi T)$, et sa fonction de scindage~$\exp(\pi (T^p-T))$, ou encore l’exponentielle d’Artin-Hasse~$\exp(T+T^p/p+T^{p^2}/p^2+\ldots)$ (cf.~\cite[VII~§2]{Robert}). Nous considérons ici les exponentielles de polynômes nuls en l’origine. Ce qui revient essentiellement aux équations différentielles linéaire ordinaires homogènes du premier ordre à coefficient polynomial. C’est dans ce contexte qu’Andrea Pulita a su réorganiser des travaux antérieurs de Philippe~Robba, en introduisant notamment des méthodes à base de vecteurs de Witt, dans~\cite{Pulita} (cf.~\cite[I Résumé~§1.1, et II~§1.1]{PulitaThese}). Gilles Christol a traduit ces travaux en un algorithme de calcul de la fonction rayon de convergence pour ces équations différentielles (\cite{Christol}). Nous présentons ici, avec simplification, quelques points clés et quelques conséquences des travaux de Pulita. L’éclaircissement que nous apportons vient de lumière que nous a apporté la relecture de la théorie de Cartier~(\cite{Cartier,Cartier3} et~\cite[§19]{HazeLectures}, et des simplifications rendues possibles par la riche  théorie des vecteurs de Witt (\cite{Haze}). Nous aboutissons à une formule explicite et directe~\eqref{formulerayon} de la fonction~$\mathrm{RoC}$ de~\cite{Christol}. La section
~§2 de~\cite{Christol} n’est pas non plus utilisée; nous la remplaçons par le Théorème~\ref{thm1} que nous démontrons directement, et un résultat datant au plus tard de~\cite{DworkRobbaEffective} (ou par l’annexe~\ref{annexe}).

Notre approche est originale, en ce sens au moins qu’elle repose entièrement sur la théorie des vecteurs de Witt, et qu’étonnamment, elle n’utilise plus de concepts provenant de la théorie des équations différentielles $p$-adiques. Cette approche se révèle remarquablement puissante, et d'autant plus à mesure que l’on  en appelle aux nombreux outils de la riche théorie des vecteurs de Witt. 

Notre exposition repose toutefois sur les travaux de Pulita. Néanmoins nous présentons des démonstrations complètes, sauf rappels à la théorie classique des vecteurs de Witt. (la référence à~\cite[Lemme~1.5]{Matsuda} est rapidement démontrée.) Cet article peut servir comme nouvelle manière d’introduire la théorie de Pulita, voire d’alternative, par voie purement algébrique, aux démonstrations originales.

\subsection*{Plan et guide de lecture} Les énoncés principaux sont les Théorèmes~\ref{thm1},~\ref{thm2} et~\ref{thm3}. Ils sont énoncés de manière la plus élémentaire possible. Nous proposons également quelques généralisations et raffinements aux numéros~\ref{thm2gen}, \ref{thm1gen}, \ref{raffinements} et  à l’Annexe~\ref{annexe}. La mise en contexte, notamment les notations et les définitions sont introduites au fur et à mesure, ce qui est permis par la brièveté de l’article. L'article est organisé pour une lecture linéaire, les notions étant introduites progressivement; ou une lecture par section, de par les rappels aux définitions et notations utilisées. Les preuves sont essentiellement complètes, détaillées, et immédiates dans le contexte utilisé (notamment celui des vecteurs de Witt).
Nous évoquons quelques développements dans la section~\ref{developpements} et détaillons quelques exemples dans la l’annexe~\ref{exemples}. Les annexes fournissent des compléments. 
\paragraph{}
Nous espérons dans un travail ultérieur utiliser le point de vue élaboré ici, ce afin de réaliser l’étude~$p$-adique des sommes d’exponentielles, dans une généralité plus grande que d’habitude, englobant notamment les sommes de~\cite{Morofushi}, et éventuellement leur généralisation à plusieurs variables. Nous obtiendrons une détermination exacte des fonctions~$L$ (et en particulier de leur degré) par interpolation. Plus précisément, nous étendrons dans cette généralité les résultats de~\cite{Terasoma}, sans hypothèse de généricité. Notre méthode étant en outre explicite et algorithmique. Nous retrouvons en particulier la conjecture de~\cite{Loeser}, et montrons même comment s’en passer, la dépasser. Nous répondons à une question de~\cite{LeStum} et de~\cite{Robba}.
\paragraph{}
La section~\ref{sec1} se résume à énoncer le Théorème~\ref{thm1} sur lequel se base notre construction. Il est formulé abstraitement en termes de la théorie des vecteurs de Witt. Sa démonstration est indépendante du reste de l’article; nous la renvoyons à la section~\ref{sectionpreuve}. La section~\ref{section2} développe les conséquences du Théorème~\ref{thm1} en termes de rayon de convergence. La section~\ref{secalgo} explique comment en déduire une formule explicite, et un algorithme de calcul du rayon de convergence. La section~\ref{developpements} détaille la variation de l’algorithme qui permet de calculer la fonction rayon de convergence au sens de~\cite{Christol}. Cette section poursuit par quelques développements quant aux informations fournies par notre méthode, mais nous touchons là des questions qui seront traitées ailleurs. La section~\ref{exemples} conclut en détaillant quelques exemples d’application de l’algorithme qui peuvent être traités à la main. Dans l’annexe~\ref{annexe} nous ajoutons un exemple d’application de nos méthodes en revenant sur un résultat de~\cite{DworkRobbaEffective} auquel nous avons fait appel, et, dans le cadre auquel nous l’avons appliqué, nous le redémontrons et le précisons par des estimations fines de coefficients. La dernière annexe donne une réduction qui permet d'étendre les résultats de cet article dans le contexte plus général de~\cite{Pulita}.

\section{Noyaux de Frobenius itérés sur les entiers}\label{sec1}
Soit~$p$ un nombre premier. 
\paragraph{}
	Pour tout anneau~$R$, notons~$W(R)$ l’anneau des vecteurs de Witt~$p$-typiques sur~$R$ (\cite[§4 Définition~1]{BBKAC8}). Notons~$F$ et~$V$ son endomorphisme de Frobenius et son opérateur additif de décalage respectivement (\cite[§5 Proposition~3]{BBKAC8}). Pour tout entier~$d$ dans~$\mathbf{Z}_{\geq 1}$, nous considérons le noyau~${}_dW(R)$ de~$F^d$, et le conoyau~$W_d(R)$ de~$V^d$. Le premier est manifestement un idéal de~$W(R)$ et le second est un anneau quotient~(\cite[§6~(36-
37)]{BBKAC8}). Ce dernier étant aussi  connu sous le nom d’anneau de vecteurs de Witt «~tronqués~», ou «~de longueur finie~» (\cite[§2]{Haze}). D’après l’identité~$V^d(a)\times b=V^d(a\times F^d(b))$, de~\cite[§6~(37)]{BBKAC8}, la structure de~$W(R)$-module de l’idéal~${}_dW(R)$ passe au quotient et définit une structure de~${W}_d(R)$-module.

\paragraph{} Une \emph{extension ultramétrique}~$K/\QQ_p$ désigne une extension de corps à laquelle on prolonge la valeur absolue de~$\QQ_p$. Il ne sera pas nécessaire ici de supposer~$K$ complet.\footnote{Peut-être suffit-il même d’une extension ultramétrique de~$\QQ$, relativement à une norme~$p$-adique.} L’anneau des \emph{entiers} désigne la boule unité fermée.
\begin{theo}\label{thm1} Soit~$R$ l’anneau des entiers d’une extension ultramétrique~$K$ de~$\QQ_p$. Si~$R$ contient une racine de l’unité~$\zeta$ d’ordre~$p^{d+1}$, alors~${}_dW(R)$ est un ${W}_d(R)$-module libre de rang~$1$, et donc un idéal principal de~${W}(R)$, avec comme générateur le vecteur de Witt~$w_d$ de composantes fantômes (\cite[§4 Définition~1]{BBKAC8}, ou~\ref{fantome} \emph{infra}) 
\begin{equation}\label{fantomepulita}(\zeta-1,\zeta^p-1,\ldots,\zeta^{p^i}-1,\ldots ).\end{equation}
\end{theo}
\noindent Ce théorème sera démontré dans la section~\ref{sectionpreuve}, et généralisé au numéro~\ref{thm1gen}.

\paragraph{}On remarquera que la suite des composantes fantômes stationne à~$0$ à partir du terme du $(d+1)$-ième terme.\footnote{Le théorème vaut aussi dans la généralité de~\cite[Définition~2.2, Remarque~2.3, p.510]{Pulita}: remplaçant~$\zeta-1$ par un point de torsion d’ordre~$p^{d+1}$ d’un groupe de Lubin-Tate isomorphe au groupe trivial, et la suite~\eqref{fantomepulita} par~\cite[Définition~2.2, p.510]{Pulita}, la suite récurrente construite par multiplication par~$p$. À~$d$ fixé, il est même possible de considérer un groupe de Lubin-Tate non trivial, mais satisfaisant une condition de la forme~\cite[Théorème~2.1~p.512, Théorème~2.5~4. p.518]{Pulita}} Une autre représentation de~$w_d$ est~\eqref{exppulita} \emph{infra}.

%

\section{Critère de solubilité-intégralité}\label{section2} Soit~$R$ l’anneau des entiers d’une extension ultramétrique~$K$ de~$\QQ_p$. Rappelons que l’exponentielle d’Artin-Hasse permet de construire un homomorphisme injectif (\cite[Définition~2.3, Remarque~2.4, p.511-512]{Pulita}, \cite[exercices~43.~a),b), 58.~d)]{BBKAC8})\footnote{Ces deux références diffèrent d’un signe dans le choix de l’indéterminée.}
\begin{equation}\label{AH}	AH:(W(K),+) \to \Lambda(K):=(1+TK[[T]],\times).\end{equation}
dont l’image est formée des séries dont le logarithme est une série~\emph{$p$-typique}. On entend par là une série~$a_0T+a_1T^p+\ldots+a_dT^{p^d}+\ldots$ de~$K[[T]]$ dont les seuls monômes non nuls sont de degré une puissance de~$p$.

\paragraph{}\label{integralite} Rappelons qu’un vecteur de Witt~$v$ de~$W(K)$ est entier~(c.-à-d., est dans~$W(R)$) si et seulement si~$AH(v)$ est à coefficients entiers (dans~$\Lambda(R)$, à coefficients dans~$R$). Cela découle de~\cite[Exercice~58.~c)]{BBKAC8} vu que~$R$ est une~$\mathbf{Z}_{(p)}$-algèbre.

\paragraph{}\label{fantome} On peut utiliser comme définition que les \emph{composantes fantômes} du vecteur de Witt $p$-typique~$v$ se déduisent des coefficients~$a_i$ comme étant la suite des~$-p^i\cdot a_i$. (Pour leur convention de signe, voir \cite[exercices~58.~c)]{BBKAC8} et voir aussi~\cite[exercices~39.~d), 40.~h)]{BBKAC8}.)\footnote{Il semble que la convention de signe de~\cite{BBKAC8} présente l’avantage de permettre des formules uniformes (mais toutes présentant des signes~«~$-$~») entre le cas~$p$ pair, et le cas des~$p$ impairs.} 
\paragraph{}\label{fantomeFV} Les applications~$F$, et~$V$ sont données en termes des composantes fantômes~$(a_0,a_1,\ldots)$ par
\begin{equation}\label{fantomeFV}(a_0,a_1,\ldots)\stackrel{F}{\mapsto}(a_1,a_2,\ldots),\text{ et }(a_0,a_1,\ldots)\stackrel{V}{\mapsto} (0,pa_0,pa_1,\ldots)\text{ respectivement,} \end{equation}
et le produit des vecteurs de Witt devient le produit composante par composante. (ou~«~produit de~Hadamard~»)
\paragraph{}\label{numeromatsuda} Le vecteur de Witt~$w_d$ du Théorème~\ref{thm1} correspond à la série
\begin{equation}\label{exppulita}AH(w_d)=\exp\left( -(\zeta-1)\cdot T-(\zeta^p-1)\cdot\frac{T^p}{p}-\ldots-(\zeta^{p^d}-1)\cdot\frac{T^{p^d}}{p^d}\right).\end{equation}
En termes de l’exponentielle de Artin-Hasse classique
$$e_{AH}(T):=AH(1)=\exp\left( -T-{T^p}/{p}-\ldots-{T^{p^d}}/{p^d}\right),$$
on récrit, suivant~\cite[Lemme~1.5]{Matsuda},
\begin{equation}\label{FormuleMatsuda}	AH(w_d)=e_{AH}(\zeta T)\cdot e_{AH}(-T).\end{equation}
Comme~$e_{AH}(T)$ est à coefficients dans~$\mathbf{Z}_{(p)}$ (voir par ex.~\cite[VII~§2.2]{Robert}), il suit que~$AH(w_d)$ est à coefficients dans~$\mathbf{Z}_{(p)}[\zeta]$.  

 La construction:~$\exp(P(T))\mapsto\exp(\widetilde{P}(T))$, de l’énoncé qui vient, correspond au produit de Hadamard par l’inverse de~$AH(w_d)$.

\paragraph{}
En termes de séries, le Théorème~\ref{thm1} a la conséquence suivante.
\begin{theo} \label{thm2}	Soit~$R$ l’anneau des entiers d’une extension ultramétrique~$K$ de~$\QQ_p$. Introduisons~$\zeta$ comme racine de l’unité d’ordre~$p^{d+1}$ dans une extension finie de~$K$.

Soit~$P(T)=a_0T+a_1T^p+\ldots+a_dT^{p^d}$ un polynôme~$p$-typique de degré au plus~$p^d$ à coefficients dans~$K$.

Alors nous avons équivalence entre les propriétés suivantes:
\begin{enumerate}	
\item Le rayon de convergence de la série~$\exp(P(T))$ est au moins~$1$;
\item \label{condition2}  La série~$\exp(P(T))$ appartient à~$\Lambda(R)$ (tous ses coefficients sont entiers);
\item Si l’on pose~$\tilde{P}(T)=\frac{a_0}{\zeta-1} T+\frac{a_1}{\zeta^p-1}T^p+\ldots+\frac{a_d}{\zeta^{p^d}-1}a_dT^{p^d}$, alors la série~$\exp(\tilde{P}(T))$ a ses premiers coefficients dans~$R[\zeta]$, jusqu’au degré~$p^d$ inclus. 
\end{enumerate}
\end{theo}
Il est instructif d’expliciter le cas~$d=0$. On retrouve bien que l’exponentielle a pour rayon de convergence~$\abs{\zeta-1}$. D’autres exemples sont détaillés au numéro~\ref{exemples}. 
 L’équivalence entre les deux premiers points est déjà connue depuis au plus tard~\cite[Théorème~4.3, pour~$n=1$, avec~$u_1=\exp(P(T))$]{DworkRobbaEffective}\footnote{La condition d’inversibilité du wronskien revient à l’inversibilité des valeurs prises par~$e(T):=\exp(P(T))$ dans son disque de convergence ouvert. Si, par l’absurde,~$e$ s’annule en un  point~$t$ de ce disque, alors nous disposons de deux solutions locales non identiquements nulles à l’équation différentielle~\eqref{equadiff} ordinaire d’ordre~$1$: la solution~$e(T)$ et la solution qui vaut~$1$ en~$t$. Mais la seconde ne peut être proportionnelle à la première. Or ce doit être le cas pour l’ordre~$1$.}. Nous la redémontrerons dans l'annexe~\ref{annexe}. L’utilité de ce théorème est manifeste si l’on sait que la première propriété est une question classique en théorie des équations différentielles $p$-adiques (la solubilité de~\eqref{equadiff} \emph{infra}),  et si l’on remarque que le second critère, connu, nécessite d’aborder une infinité de coefficients, tandis que le troisième se vérifie par un algorithme immédiat. En outre, ce théorème s’énonce élémentairement, sans faire appel aux vecteurs de Witt sous-jacents. L’algorithme évoqué ne nécessite pas de calculer des composantes de vecteurs de Witt.
\begin{proof}[Preuve du Théorème~\ref{thm2}]
Le polynôme~$P(T)$ provient, via l’application~$\log\circ AH$, d’un élément~$v$ de~$W(K)$. La série~$\exp(P(T))$ s’écrit alors comme l’élément~$AH(v)$ de~$\Lambda(K)$. Or~$P(T)$ est nul au-delà du degré~$p^d$. D’après~\eqref{fantomeFV}, cela équivaut à~$F^{d+1}(v)=0$. Ainsi~$v$ est dans~${}_dW(K)$.

La condition~$2$ revient alors à affirmer que~$AH(v)$ appartient en fait à~$\Lambda(R)$. C’est-à-dire (cf.~\ref{integralite}) que~$v$ est dans~$W(R)$. Comme~$\zeta$ est entier sur~$R$, et~$R$ intégralement clos dans~$K=R[1/p]$, nous avons~$R[\zeta]\cap K=R$, d’où~$\Lambda(R[\zeta])\cap \Lambda(K)=\Lambda(R).$ La condition~\ref{condition2} équivaut donc à ce que l’élément~$AH(v)$ appartienne à~$\Lambda(R[\zeta])$. Ou bien à ce que~$v$ soit dans dans~${}_{d+1}W(R[\zeta])$.

Appliquons le Théorème~\ref{thm1} à~$R[\zeta]$ et à~$\zeta$. Nous obtenons un isomorphisme
$$W_{d+1}(R[\zeta])\xrightarrow{x\mapsto x\cdot w_d} {}_{d+1}W(R[\zeta])$$
qui s’étend par la même formule en un isomorphisme~$W_{d+1}(K[\zeta])\to {}_{d+1}W(K[\zeta])$. Nous pouvons écrire~$v=x\cdot w_d$. En termes de composantes fantômes, on vérifie que la série~$p$-typique associé à~$x$, au sens du début de~§\ref{section2}, n’est autre que le polynôme~$\tilde{P}(T)$.

Le Théorème~\ref{thm1} nous donne l’équivalence: 
$$v\in {}_{d+1}W(R[\zeta]) \Leftrightarrow x\in W_{d+1}(R[\zeta]).$$ Il reste à décider si~$x$, qui est dans~$W_{d+1}(K[\zeta])$, appartient à~$W_{d+1}(R[\zeta])$. On utilise le lemme suivant pour~$R[\zeta]$ et~$K(\zeta)$.
\end{proof}
\begin{lemme}[cf.~{\cite[§A.7, p.164]{Kay}}]\label{Lemmecritere} Pour tout~$x$ dans~$W(K)$, on a équivalence entre:
\begin{enumerate}
\item La classe~$x+V^dW(K)$ de~$x$ dans~$W_d(K)$ est entière. (appartient à~$W_d(R)$)
\item La série~$AH(x)$ a des coefficients entiers jusqu’au degré~$p^d$ inclus.
\end{enumerate}
\end{lemme}
Ce lemme semble essentiellement connu. Il découlera aisément, au numéro~\ref{universels} qui suit, de la théorie de l’anneau des vecteurs de Witt « universels », ou « généralisés » (\cite[§26]{Mumford-Bergman}), et de ses troncations. (Voir \cite[§§14.15--14.25]{Haze},\cite[§1, §3]{Cartier3}) L’auteur remercie L.~Hesselbot pour lui avoir, diligemment, pointé la référence~\cite[§A.7, p.164]{Kay}\nocite{KayErr}. Voir aussi~\cite[§1]{Larspublie} qui détaille ce dont nous aurons besoin; notamment~\cite[1.1, 1.14, 1.15, 1.16]{Larspublie}. Des considérations analogues semblent évoquées dans~\cite[§4.1]{Manin} et semblent paraître dans~\cite[§2, Lemmes 2.1, 2.2]{Katz}. Les considérations qui suivent~(\ref{universels} et~\ref{remarquable}) vont nous permettre de généraliser les Théorèmes~\ref{thm1} et~\ref{thm2} dans le contexte «non nécessairement~$p$-typique~», aux numéros qui suivront. 

Nous suivons~\cite[1]{Larspublie}.
\paragraph{\emph{Vecteurs de Witt universels tronqués}}\label{universels}
\newcommand{\W}{\mathbb{W}}
\newcommand{\WD}{\mathbb{W}_{\{1;\ldots;D\}}}
Il existe un foncteur qui, pour tout anneau~$R$, définit un anneau~$\W(R)$, naturellement isomorphe~$\Lambda(R)$, et dont~$W(R)$ est naturellement un quotient. Pour tout entier~$D$, on considère le quotient
\begin{equation}\label{quotientmoddegre}\WD(R)\simeq\left.\Lambda(R) \middle/\left(1 + T^{D+1}R[[T]]\right)\right. .\end{equation}
Lorsque~$R$ est une~$\mathbf{Z}_{(p)}$-algèbre, l’application~$AH$ fait de~$W(R)$ un facteur direct de~$\W(R)$. En composant avec~$V_n:e(T)\mapsto e(T^n)$ on obtient même une factorisation
\begin{equation}\label{factorisationU}\W(R)=\prod_{p\nmid n} V_n W(R).\end{equation}

Cette factorisation~\eqref{factorisationU} induit au quotient~\eqref{quotientmoddegre} une factorisation
\begin{equation}\label{factorisationtronquee}\prod_{p\nmid n} W_{d_n}(K)\simeq \left.\Lambda(K) \middle/\left(1 + T^{D+1}K[[T]]\right)\right. ,\end{equation}
où~$d_n$ est l’entier maximal tel que~$np^{d_n}\leq D$, et se calcule comme l’arrondi entier par défaut
\begin{equation}\label{arrondi} d_n=\floor{\log_{p}{\frac{D}{n}}}.\end{equation}

Le fait important étant la fonctorialité par rapport la~$\mathbf{Z}_{(p)}$-algèbre~$R$. Si~$R$ est une~$\mathbf{Z}_{(p)}$-algèbre sans torsion,
posons~$K=R[1/p]$. Alors un élément~$x$ de~$W_{d_1}(K)$ provient de~$W_{d_1}(R)$ si et seulement son image dans~$\WD(K)$ provient de~$\WD(R)$. Autrement dit ci la série tronquée qui lui correspond pa~\eqref{quotientmoddegre} a ses~$D$ coefficients dans~$R$.

Le Lemme~\ref{Lemmecritere} s’en déduit en choisissant~$D=p^d$.

\paragraph{}\label{remarquable}Dans~$\log(\Lambda(K))$, la factorisation~\eqref{factorisationU} se déduit de la réécriture d’une série 
\begin{equation}\label{factorisation} P(T)= \sum_i a_iT^i = \sum_{p\nmid n} P_n(T^n),\text{ où }P_n(T^n)=\sum_d a_{np^d}T^{np^d},\end{equation}
en termes de séries $p$-typiques~$P_n(T)$. La factorisation correspond à
$$\exp(P(T)\mapsto(\exp(P_n(T^n)))_{n\in\left\{k\geq 1\,\middle|\,p\nmid k\right\}}.$$
 En particulier,~$\exp(P(T))$ est à coefficients dans~$R$ (resp. jusqu’au degré~$D$) si et seulement si il en est de même de chacun de ses «~facteurs~$p$-typiques~» $\exp(P_n(T^n))$.

\paragraph{\textit{Généralisation au cas non nécessairement~$p$-typique}}

%
En utilisant ainsi les vecteurs de Witt universels, on peut montrer la généralisation suivante du Théorème~\ref{thm2} où l’on ne suppose plus que soit~$p$-typique le polynôme~$P$. On choisit~$P(T)=\sum_{i=1}^D a_1 T^i$, d’un certain degré~$D\geq 1$ et toujours de terme constant nul.

Au prix d’une formule un peu plus alambiquée pour définir~$\widetilde{P}(T)$. Utilisons~\eqref{arrondi} pour définir~$d=d_D$. Soit~$\zeta$ comme dans le Théorème~\ref{thm1}, pour ce~$d$. Notons~$\zeta_d=\zeta$, et définissons~$\zeta_{d-i}=\zeta^{p^i}$.  Pour tout degré~$1\leq n \leq D$, reprenons la définition~\eqref{arrondi}. Nous pouvons poser
\begin{equation}\label{Puniv}\widetilde{P}(T)=\frac{a_1}{\zeta_{d_1}-1} T+\frac{a_2}{\zeta_{d_2}-1}T^2+\ldots+\frac{a_D}{\zeta_{d_D}-1}T^{D}.\end{equation}

\paragraph{Exemple} Si~$D=17$ et~$p=2$, les~$d_i$ sont~$4$, $3$, $2$, $2$, $1$, $1$, $1$, $1$ et $0$ pour~$9\leq i\leq17$, les~$\zeta_{d_i}$ sont d’ordre~$32,16,8,8,4,4,4,4,2,2,\ldots$, les premiers~$\zeta_i$ peuvent être choisis
$$-1,~i,~\frac{\sqrt{2}+\sqrt{2}i}{2},\frac{\sqrt{2+\sqrt{2}}+i\sqrt{2-\sqrt{2}}}{2},\frac{\sqrt{2+\sqrt{2+\sqrt{2}}}+i\sqrt{2-\sqrt{2+\sqrt{2}}}}{2}=e^\frac{\pi*i}{16}.$$
L’analogue de~\eqref{exppulita} sera~\eqref{exppulitaglobal}, page~\pageref{exppulitaglobal}.
\paragraph{}\label{thm2gen} La variante suivante du Théorème~\ref{thm2} se déduit par un simple jeu de traduction autour du numéro~\ref{remarquable}.
\begin{coro}\label{thm2bis}	Le Théorème~\ref{thm2} vaut pour un polynôme~$P(T)=\sum_{i=1}^D a_1 T^i$ qui n’est plus supposé~$p$-typique: en posant~$D=np^d$ avec~$\pgcd{n}{p}=1$; en choisissant~$\zeta$ une racine de l’unité d’ordre~$p^{d+1}$; en définissant~$\widetilde{P}(T)$ par~\eqref{Puniv}; et, dans la dernière condition, en requérant l’intégralité de~$\exp(\widetilde{P}(T))$ jusqu’au degré~$D$ inclus.
\end{coro}
\begin{proof}Utilisons la décomposition~\eqref{factorisation}. D’après la factorisation~\eqref{factorisationU} et la remarque~\eqref{remarquable}, l’intégralité de~$\exp(P(T))$ revient à l’intégralité simultanée de chacune des~$\exp(P_n(T^n))$. On utilise alors le théorème~$2$ pour chacun des~$P_n(T)$ avec~$d=d_n$. Celà revient à l’intégralité des~$\exp(\widetilde{P_n}(T))$ jusqu’au degré respectif~$p^d_n$. Où encore à l’intégralité de~$\exp(\widetilde{P_n}(T^n))$ jusqu’au degré~$np^{d_n}$. Où encore à l’intégralité de~$\exp(\widetilde{P_n}(T^n))$ jusqu’au degré~$D$, vu que c’est une série en~$T^n$.

Par construction, le polynôme~$\widetilde{P}(T)$ de~\eqref{Puniv} s’écrit~$
\sum_{p\nmid n}\widetilde{P_n}(T)$. En vertu du numéro~\ref{remarquable}, l’intégralité simultanée des~$\exp(\widetilde{P_n}(T^n))$ jusqu’au degré~$D$ revient à l’intégralité de~$\widetilde{P}(T)$ jusqu’au degré~$D$.
\end{proof}
\paragraph{}\label{thm1gen} De la même manière, nous avons l’analogue suivant du Théorème~\ref{thm1}.
\begin{coro} Sous les hypothèse du Théorème~$1$ concernant~$R$, les exponentielles de polynômes de degré au plus~$D$ qui sont à coefficients entiers décrivent l’idéal de~$\Lambda(R)$ engendré par
\begin{equation}\label{exppulitaglobal}\exp\left((\zeta_{d_1}-1)X+(\zeta_{d_2}-1)\frac{X^2}{2}+\ldots+(\zeta_{d_D}-1)\frac{X^D}{D}\right).\end{equation}
Cet idéal est un module libre de rang un sur le quotient~\eqref{quotientmoddegre}.
\end{coro}

\paragraph{}\label{raffinements} Mentionnons, pour information, et sans preuve les raffinements suivants.
\begin{prop} Dans le Théorème~\ref{thm2} (resp. le Lemme~\ref{Lemmecritere}), il suffit de vérifier l’intégralité de~$\exp(\widetilde{P}(T))$ (resp. de~$AH(x)$) qu’aux degrés~$1,p,\ldots,p^d$.
\end{prop}
Pour des polynômes non nécessairement $p$-typiques, mais lacunaires, l’analogue suivant peut être pertinent.
\begin{prop} Dans le Corollaire~\ref{thm2bis}, soit~$J_P$ le monoïde multiplicatif engendré par~$p$ et par les degrés des monômes de~$P$. Alors il suffit de tester l’intégralité de~$\exp(\widetilde{P}(T))$ aux degrés pris dans~$J_P\cap\{1;\ldots;D\}$.
\end{prop}

%

\section{Algorithme de calcul du rayon de convergence}\label{secalgo} Considérons une équation différentielle ordinaire homogène du premier ordre algébrique et sans pôles sur la droite affine sur~$K$. On entend par là
\begin{equation}\label{equadiff} y^\prime=L(T)\cdot y\end{equation}
pour un certain coefficient~$L(T)$ dans~$K[T]$. Soit~$P(T)$ la primitive de~$L(T)$ sans terme constant. Alors une solution formelle à l’origine est donnée par la série~$\exp(P(T))$. Les autres solutions formelles en sont les multiples scalaires. Nous souhaitons déterminer le rayon de convergence de~$\exp(P(T))$. Ce qui s’obtient comme conséquence immédiate du Théorème~\ref{thm2}. Nous utilisons la formule~\eqref{Puniv} concernant la construction de~$\tilde{P}(T)$.



\begin{theo}\label{thm3}Avec les notations ci-dessus, soit~$D$ le degré de~$P(T)$. Soit~$\tilde{e}(T)=1+\sum_{i=1}^{p^d}\tilde{a}_iT^i$ le polynôme déduit de la série~$\exp(\tilde{P}(T))$ en ne conservant que les termes de degré au plus~$D$ inclus.

Alors le rayon de convergence~$\rho$ de~$\exp(P(T))$ est
\begin{equation}\label{formuleintegraliterayon}\rho=\max\left\{\abs{\lambda}\in\mathbf{R}~\middle|~\max_{1\leq i\leq D} \abs{\tilde{a}_i}\cdot\abs{\lambda}^i\leq 1\right\}.\end{equation}
\end{theo}
Nous entendons par \emph{polygone de Newton dual} la fonction~$\log{\abs{\lambda}}\mapsto \log N(\abs{\lambda})$, où~$N(\abs{\lambda})$
est la norme de Gauß~$\max_{0\leq i\leq p^d} \abs{\tilde{a}_i}\abs{\lambda}^i$ (pour~$\tilde{a}_0=1$)  de~$\tilde{e}(T)$ au rayon~$\abs{\lambda}$. Elle a pour épigraphe le polygone convexe qui se déduit par transformation de Legendre du polygone de Newton de~$P(T)$ défini classiquement (\cite[§6.4]{Gouvea}). Nous pouvons reformuler en ces termes la caractérisation du rayon~$\rho$.
\begin{coro}  Comme précédemment, soit~$\rho$ le rayon de convergence d’une solution formelle non nulle à l’origine  de~\eqref{equadiff}. Ce rayon correspond à la fin de la première pente, nulle, du polygone de Newton dual de~$\tilde{e}(T)$. C’est également le plus petit rayon d’une racine de~$\tilde{e}(T)$; il est donné par la formule
\begin{equation}\label{formulerayon}-\log\rho=\max_{1\leq i\leq D}\frac{1}{i}\cdot{\log\abs{\widetilde{a_i}}}.\end{equation}
\end{coro}

\begin{proof}
Le rayon de convergence de~${\exp(P(T))}$ est la borne supérieure des~$\abs{\lambda}$, relatifs aux~$\lambda$ dans~$\mathbf{C}_p^\times$ tels que la série~$\exp(P(\lambda\cdot T))$ a rayon de convergence au moins~$1$.

D’après le dernier critère du Théorème~\ref{thm2}, cela revient à l’intégralité de~$\tilde{e}(\lambda\cdot T)$. Autrement dit à ce que sa norme de Gauß, au rayon unité,~$\max_{i=0}^d \abs{\tilde{a}}\abs{\lambda}^i$ soit bornée par~$1$.

Le théorème en découle directement.
\end{proof}
\paragraph{}  Pour calculer~$\tilde{e}(T)$, il suffit de développer jusque l’ordre~$D$ une solution formelle de
$$y^\prime=\tilde{L}(T)\cdot y,$$
où~$\tilde{L}(T)=\frac{d}{dT}\widetilde{P}(T)$.
\paragraph{} Il semble que le polygone de Newton de~$\widetilde{e}(T)$ donne la fonction rayon de convergence sur le segment ....
Autrement dit, pour un point générique~$a$ de rayon~$r$, le rayon~$\rho(a)$ d’une solution formelle de~\eqref{equadiff} au point~$a$ vérifie
$$\min\{\abs{a};r(a)\}=\abs{\tilde{e}(a)}.$$
De manière équivalente
$$r(a)=\left\{
\begin{array}{ll}
\rho&\text{ si $\abs{a}\leq\rho$},\\
\abs{\tilde{e}(a)}&\text{ sinon.}
\end{array}\right.$$

Dans un tel cas, la somme des multiplicités des racines de~$\tilde{e}(T)$ de rayon~$\rho$ s’interprète comme l’indice de l’équation différentielle~\eqref{equadiff} au rayon~$\rho$.

\section{Preuve du Théorème~\ref{thm1}}\label{sectionpreuve} Afin de démontrer le théorème~\ref{thm1} commençons par quelques observations. Rappelons que~$\zeta$ est une racine d’ordre~$p^{d+1}$ de l’unité. Notons~$\zeta_d=\zeta$, et définissons~$\zeta_{d-i}=\zeta^{p^i}$. Ainsi, pour~$0\leq i\leq d$, la racine de l’unité~$\zeta_i$ est d’ordre~$p^{i+1}$.

 Soit~$w_d$ le vecteur de Witt de~$W(K)$ de composantes fantômes~\eqref{fantomepulita} défini dans le Théorème~\ref{thm1}. Soient également~$w_{d-i}=F^i(w_d)$, qui a composantes fantômes
 $$(\zeta_{d-i}-1,\zeta_{(d-i)-1}-1, \zeta_{(d-i)-2}-1,\ldots).$$
\begin{lemme}\label{lemmeimbrication} Le polynôme de Lubin-Tate~$(X+1)^p-1$ du groupe multiplicatif admet le cycle
\begin{equation}\label{cycle} \zeta_d-1\mapsto\zeta_{d-1}-1\mapsto\ldots\mapsto \zeta_1-1\mapsto0\mapsto 0 \mapsto \ldots.\end{equation}
Ce polynôme s’écrit~$X\cdot H(X)$ où~$H$est à coefficients entiers. Pour~$i$ dans~$\{1;\ldots;d\}$, on a les identités~$w_{j-1}=w_j\cdot H(w_j)$. Il s’ensuit l’imbrication d’ideaux de~$W(R)$ suivante
\begin{equation} \label{imbrication} (w_d) \supseteq (v_{d-1}) \supseteq \ldots \supseteq (w_1) \supseteq (0).\end{equation}
\end{lemme}
Ce lemme est sans difficulté et reprend partie de \cite[Proposition~2.1, p.511]{Pulita}.

\begin{proof}[Démonstation du Théorème~\ref{thm1}] Soit~$w_d$ le vecteur de Witt de~$W(K)$ de composantes fantômes~\eqref{fantomepulita} du Théorème~\ref{thm1}. Comme la suite~\eqref{fantomepulita} n’a que ses~$d$ premières composantes non nulles (cf.~\eqref{cycle} \emph{supra}), il est immédiat que~$w_d$ appartient à~${}_dW(K)$ (cf.~\ref{fantome}~\eqref{fantomeFV}). Tout revient à démontrer, d’une part que~$w_d$ appartient à~${}_dW(R)$, de sorte que l’application
\begin{equation}\label{isomorphisme} \phi^d_R:W_d(R)\xrightarrow{\lambda\mapsto\lambda\cdot w_d} {}_dW(R)\end{equation}
soit bien définie, et de montrer d’autre part qu’il s’agit d’un isomorphisme. 

\begin{proof}[Preuve de l’intégralité]  L’intégralité de~$w_d$ revient à celle de la série correspondante
\begin{equation}\label{expRobba} \exp\left((\zeta-1)T+(\zeta^p-1)T^p+\ldots+(\zeta^{p^d}-1)T^{p^d}\right)\end{equation}
qui est une exponentielle de Robba au sens de \cite[§0.2]{Pulita}. Notre cas particulier est couvert par le numéro~\ref{numeromatsuda},~p.\pageref{numeromatsuda} (cf.~\cite[Lemme~1.5]{Matsuda}). Pour des méthodes plus générales, voir~\cite{Pulita}.
\end{proof}

\begin{proof}[Preuve de l’injectivité]
L’injectivité est manifeste sur l’écriture de~\eqref{isomorphisme} en composantes fantômes (cf.~\ref{fantome})
\begin{equation}\label{isofantome}\left(\phi_1,\ldots,\phi_d\right)\mapsto\left({\phi_1}\cdot({\zeta_d-1}),\ldots,{\phi_d}\cdot({\zeta_1-1}), 0,0,\ldots\right)\end{equation}
 Par surcroît,~$\phi^d_R$ s’étend en un isomorphisme~$W_d(K)\xrightarrow{\phi^d_K} {}_dW(K)$. 
\end{proof} 

Il reste à montrer la surjectivité, à laquelle nous consacrons le numéro suivant.
\end{proof}
\paragraph{}\label{gamma0} Utilisons la notation~$\gamma_0(t)$ pour le vecteur de Witt~$(t,0,0,\ldots)$, de composantes fantômes~$(t,t^p,t^{p^2},\ldots)$. Il s’agit de la courbe $p$-typique universelle de~\cite[§4]{Cartier}. Nous avons l’identité
\begin{equation}\label{identiteF}F(\gamma_0(t))=\gamma_0(t^p).\end{equation}

\begin{proof}[Preuve de la surjectivité.] Nous souhaitons montrer la surjectivité de~\eqref{isomorphisme}. Autrement dit que tout~$x$ dans~${}_dW(R)$ est atteint par~\eqref{isomorphisme}. Écrivons~$x=w\cdot v_d$ avec~$w$ dans~$W_d(K)$. Il suffit de montrer que~$w$ est entier.

La présente démonstration se fait par récurrence sur~$d$.

L’étape d’initialisation se fait pour~$d=1$. Elle se traduit par le fait bien connu suivant: la série~$\exp(\lambda\cdot T)$ n’a de coefficient entiers que si~$\lambda$ est un multiple de~$(\zeta-1)$ par en entier. Cela résulte de l’estimation classique des coefficients de la série exponentielle:~$\limsup\frac{1}{n}\abs{n!}=\sup\frac{1}{n}\abs{n!}=\abs{\zeta-1}={\abs{p}}^{1/(p-1)}$, conséquence du calcul exact de la valuation~$p$-adique
$$ \log_{\abs{p^{-1}}}\abs{n}=\sum_{d\geq1} \floor{\frac{n}{p^d}}.$$
(\cite[Problèmes 164-165, Lemme~4.5.5, p.115]{Gouvea}.)

Pour l’étape de récurrence, nous souhaitons démontrer la surjectivité de~\eqref{isomorphisme} pour un indice~$d$. Autrement dit l’intégralité de~$w$ correspondant à un~$x$ donné dans~${}_dW(R)$. Nous pouvons supposer la surjectivité de~\eqref{isomorphisme} jusqu’à l’indice~$d-1$ inclus, et pour n’importe quel corps tel que dans le Théorème~\ref{thm1}. Soit~$C$ une extension ultramétrique algébriquement close de~$K$, et notons~$S$ son anneau d’entiers. Ainsi, tout élément de~$R$ a une racine~$p^{d-1}$-ième dans~$C$, qui est nécessairement dans~$S$. 

L’étape de récurrence utilise la suite exacte
$$0\xrightarrow{}{}_{d-1}W(R)\to {}_{d}W(R)\xrightarrow{F^{d-1}}{}_1W(R)$$

 Comme~$F^{d-1}(x)$ appartient à~${}_1W(R)$, il s’écrit~$\lambda\cdot  w_1$ avec~$\lambda$ dans~$R$, d’après l’étape d’initialisation. Soit~$\tilde{\lambda}$ un entier de~$C$ solution de~${\tilde{\lambda}}^{p^{d-1}}=\lambda$. Alors~\eqref{identiteF} nous donne~$F^{d-1}(\gamma_0(\tilde{\lambda}))=\gamma_0(\lambda)$. Par ailleurs,~$F^{d-1}(w_{d})=w_1$. Par conséquent, 
$$F^{d-1}(\tilde{\lambda}\cdot w_{d})= \gamma_0(\lambda)\cdot  w_1=\lambda\cdot w_1=F^{d-1}(x).$$
Donc~$x-\gamma_0(\tilde{\lambda})\cdot w_{d}$ appartient au noyau~${}_{d-1}W(C)$ de~$F^{d-1}$. Mais~$\tilde{\lambda},w_{d}$ et~$x$ sont à coordonnées entières (dans~$S$, $\mathbf{Z}_p[\zeta]$ et~$R$ respectivement). Il s’agit donc d’un élément de~${}_{d-1}W(S)$. En appliquant l’hypothèse de récurrence, au corps~$C$, nous déduisons que~$x-\gamma_0(\tilde{\lambda})\cdot w_{d}$ s’écrit~$y\cdot w_{d-1}$ avec~$y$ dans~$W_{d-1}(S)$. Bref
$$x = \gamma_0(\tilde{\lambda})\cdot w_{d} + y \cdot w_{d-1}.$$
Mais~$w_{d-1}$ s’écrit aussi~$w_{d}H(w_{d})$ avec~$XH(X)=(X+1)^p-1$ (Lemme~\ref{lemmeimbrication}). Finalement~$x$ s’écrit~$w\cdot w_{d+1}$ avec~$w=\gamma_0(\tilde{\lambda})+y\cdot H(w_{d})$. Ainsi~$w$ est à coordonnées dans~$S$; mais aussi dans~$K$. Donc~$w$ est dans~$W_d(R)$.
\end{proof}

\paragraph{}Notre preuve de la surjectivité diffère des méthodes de~\cite{Pulita}, qui exploite le lien avec les équations différentielles~$p$-adiques. Notre preuve est plus algébrique et ne nécessite pas la complétude de~$R$. Pour l’injectivité de~\eqref{isofantome}, nous utilisons que les~$\zeta_{d-i}$ ne sont pas diviseurs de zéro. Pour la surjectivité, nous utilisons la norme ultramétrique et le fait que~$R$ est la boule unité. Et nous réutilisons que~$R$ est intégralement clos. 

\section{Fonction rayon de convergence et Questions annexes}\label{developpements} En utilisant~\eqref{formulerayon}, nous pouvons en déduire une formule «~globale~» du rayon de convergence d’une solution formelle de l’équation différentielle~\eqref{equadiff}, une formule qui s’applique en une origine indéterminée. Suivons la méthode de G.~Christol et faisons une translation de l’origine des coordonnées de~$0$ vers~$a$, où~$a$ appartient à une extension ultramétrique de~$K$. Alors nous pouvons récrire
\begin{subequations}
\begin{equation} P(T)=P(a)+\sum_{i=1}^D a_i(a)(T-a)^i,\label{nota}\end{equation}
ou de manière équivalente
\begin{equation} P(T,a):=P(T+a)=P(a)+\sum_{i=1}^D a_i(a)T^i,\end{equation}
avec les polynômes
\begin{equation}\label{formulea} a_i(a):=\sum_{k=i}^D a_k(0)\cdot\binom{i}{k}\cdot a^{k-i}.\end{equation}
Construisons
\begin{equation}\label{Pfamille}\tilde{P}(T,a)=\frac{a_1(a)}{\zeta_{d_1}-1} T+\frac{a_2(a)}{\zeta_{d_2}-1}T^p+\ldots+\frac{a_D(a)}{\zeta_{d_D}-1}a_DT^{D},\end{equation}
et formons ensuite la série~$\exp(\widetilde{P}(T,a))\in \Lambda(K[a])$, que l’on tronque au degré~$D$ relativement à~$T$, ce qui donne un polynôme
\begin{equation}\label{formuleetildefamille}\widetilde{e}(T,a)=1+\sum_{i=1}^D \widetilde{a}_i(a)T^i,\end{equation}
\end{subequations}
l’unique élément dans~$(1+TK[a]+\ldots+T^DK[a])\cap(\exp(\widetilde{P}(T,a))+T^{D+1}K[a][[T]]$.
\begin{theo}
Le rayon\footnote{Il s’agit de la fonction~$\mathrm{RoC}$ de~\cite{Christol}} de convergence~$\rho(a)$ de la série de Taylor en~$a$ solution non nulle de l’équation~\eqref{equadiff} est tel que, avec les notations~\eqref{nota} à~\eqref{formuleetildefamille} ci-dessus,
\begin{equation}\label{formuleglobalerayon}-\log\rho(a)=\max_{1\leq i\leq D}\frac{1}{i}\cdot{\log\abs{\widetilde{a_i}(a)}}.\end{equation}
\end{theo}
\paragraph{}Il s’ensuit que la «~fonction~»~$a\mapsto-\log\rho(a)$ jouit de nombreuses propriétés des fonctions de la forme «~logarithme de la norme d’un polynôme~».  (cf.~\cite[Remarque au numéro~§4.5~p.206]{Robba3})
\begin{itemize}
\item Propriétés de convexité au sens des polygones de Newton (ou de superharmonicité).
\item Elle est naturellement définie sur l’espace analytique, au sens de Berkovich, de la droite affine sur le complété de~$K$, et y définit une fonction continue.
\item Elle est déterminée, à une constante additive près, par le lieu de ses «~zéros~», comptés avec multiplicité, par une formule de Poincaré-Lelong~\cite[5.20]{BakerRumely}.  (voir~\cite{Christol} Proposition~3.3)
\item Elle se calcule par rétraction au un sous-graphe fini engendré par le support de son diviseur.
\end{itemize}

Mentionnons~\cite{Baldassarri,Pulitarayon,PulitaPoineau} pour les propriétés qualitatives des fonctions rayon de convergence dans un contexte plus général.

\paragraph{\emph{Question d’effectivité en mémoire}}
D’après~\eqref{formulea}, le polynôme~$a_i(a)$ est de degré au plus~$D-i$. On peut en déduire que le polynôme~$\widetilde{a_i}(a)$ a degré au plus~$i(D-1)$, et est à coefficients dans l’extension~$K[\zeta]$. C’est une extension de degré au plus au plus~$(p-1)p^{\floor{\log_p(D)}}$ sur~$K$. 

Ainsi, le polynôme~\eqref{formuleetildefamille} est déterminé par au plus~$D(D-1)^3 \cdot \floor{\log_p(D)}/2\leq D^3\log(D)$ coefficients pris dans~$K$.

\paragraph{\emph{Construction du diviseur}} Pour ramener la formule~\eqref{formuleglobalerayon}, qui est un maximum, au logarithme de la norme d’un unique polynôme, nous pouvons utiliser la construction suivante. Introduisons~$L=\mathrm{Frac}(K[\zeta]\{T_1;\ldots;T_D\})$ le corps des fractions de l’algèbre de Tate sur~$K[\zeta]$, éventuellement complété relativement à la norme de Gauß. Formons le polynôme
$$\widetilde{A}(a)=1+\sum_{i=1}^D (\widetilde{a_i}(a))^{D!/i} T_i\text{ dans }K[\zeta]\{T_1;\ldots;T_D\}[a].$$

Soit~$\mathrm{Div_L}(\widetilde{A})=\sum_{x\in\overline{L}} n_x \delta_x$ le diviseur de~$\widetilde{A}$, où~$\overline{L}$ est une extension algébrique et algébriquement close de~$L$, où~$n_x$ désigne le degré d’annulation de~$\widetilde{A}$ en~$x$, et où~$\delta_x$ est la masse de Dirac placée en~$x$. À chaque~$x$ nous pouvons associer un point\footnote{C’est un point de type~(1).} dans l’espace de Berkovich de la droite affine sur le complété de~$L$, mais aussi sa projection\footnote{Cette projection n’est plus nécessairement un point de type~(1), mais peut être un point de type~(2).} sur l’espace de Berkovich de la droite affine sur le complété de~$K$. Ce point est donné par
$$K[T]\xrightarrow{P(T)\mapsto\abs{P(x)}}\mathbf{R}$$
pour la valeur absolue sur~$\overline{L}$ qui prolonge celle de~$L$. Notons~$[x]$ cette projection, et formons
\begin{equation}\label{diviseur}\mathrm{Div_K}(\widetilde{A})=\sum_{x\in\overline{L}} n_x \delta_{[x]}.\end{equation}
Alors la fonction~$\rho(a)$ est déterminée à une constante multiplicative près par l’identité de séries
$$D!\cdot\Delta(-\log(\rho))=\mathrm{Div_K}(\widetilde{A})$$
sur l’espace de Berkovich de la droite affine sur le complété de~$K$. On peut en déduire que~$\rho$ est déterminée par rétraction à l’enveloppe convexe du support de~\eqref{diviseur}.  

Notons~$\pi:A_L\to A_K$ la projection entre espaces de Berkovich précédente. La fonction réelle~$\bigl|\widetilde{A}\bigr|$ sur~$A_L$ et la fonction réelle~$\rho(a)$ sur~$A_K$ sont reliées explicitement par 
$$\rho(a)=\max_{\pi(x)=a} \bigl|\widetilde{A}\bigr| (x).$$

\newcommand{\Ocal}{\mathcal{O}}
\paragraph{\emph{Question annexe sur l’indice cohomologique}} Il est possible que la mesure~$\mu=\mathrm{Div_K}(\widetilde{A})/D!$ aie l’intreprétation suivante. Soit~$A$ un affinoïde de~$A_K$, et soit~$\Ocal(A)^\dagger$ l’algèbre des fonctions surconvergentes sur~$A$. Alors~$\mu(A)$ détermine la dimension de la cohomologie de de Rham de l’équation~\eqref{equadiff} à coefficients dans~$\Ocal(A)^\dagger$.

Dans le cas d’un polynôme, la considération du polygone de Newton donne lieu à une factorisation. Plus généralement la détermination de son diviseur donne lieu à sa factorisation en facteurs irréductibles. 

On s’attend à une factorisation analogue de l’équation différentielle~\eqref{equadiff} en relation avec sa fonction rayon de convergence. Cette factorisation n’est pas étrangère à l’algorithme de~\cite{Christol}, et il semble que cela revienne au calcul des coordonnées du vecteur de Witt associé à~$\exp(\widetilde{P}(T))$.

\appendix
\section{{Exemples}}\label{exemples} \emph{Illustrons le Théorème~\ref{thm2}, via la formule~\eqref{formuleintegraliterayon}, sur quelques exemples fondamentaux.}
\makeatletter
\renewcommand{\thesubsection}{\textbf{\@Alph\c@section.\@arabic\c@subsection}}
\makeatother

 Reprenons les notations~$\zeta_i$ du début de la section~\ref{sectionpreuve}. Autrement dit~$\zeta_i$ est une racine de l’unité d’ordre~$p^{i+1}$ et les~$\zeta_i$ satisfont à la relation de compatibilité~$(\zeta_{i+1})^p=\zeta_i$. Il sera commode de poser comme abbréviation~$\pi_i=\zeta_i-1$. Rappelons que nous avons (\cite[II~§4.4]{Robert})
\begin{subequations}
\begin{equation}\abs{\pi_0}=\abs{p}^{1/(p-1)}\label{rayon0}\end{equation} 
et
\begin{equation}\abs{\pi_{i}}=\abs{\pi_0}^{1/p^i}=\abs{p}^{\left.1\middle/p^i(p-1)\right.}.\label{rayoni}\end{equation}
Retenons la relation suivante, déduite des deux précédentes,
\begin{equation}
\left.\abs{\pi_{0}}\middle/\abs{\pi_1}\right.=
\abs{p}^{\left(1\middle/(p-1)\right)-\left.1\middle/p(p-1)\right.}
=
\abs{p}^{\left.1\middle/p\right.}
.\label{rayon01}\end{equation}
\end{subequations}

\begin{example}[\emph{La série~$\exp(T)$}]
 Nous avons~$P(T)=T$ et~$d=0$ (ou~$D=1$). Alors~$\widetilde{P}(T)=\frac{1}{\pi_0}T$, et
$$\widetilde{e}(T)=1+\frac{1}{\pi_0}T\equiv\exp(\tilde{P}(T))\pmod{T^2}$$
 L’intégralité de~$\widetilde{e}(T)$ s’écrit
\begin{equation}\label{conditionexp}\abs{T}\leq\abs{\pi_0}.\end{equation}
On retrouve bien le rayon de l’exponentielle, vu~\eqref{rayon0} et~\cite[Tables, p.427]{Robert}.
\end{example}

\begin{example}[\emph{La série~$\exp(T+0\cdot T^p)$}]  Reprenons~$P(T)=T$, mais~$d=1$. Autrement dit~$P(T)$ est vu comme polynôme $p$-typique de degré au plus~$p$. Cette fois
$$\widetilde{P}(T)=\frac{1}{\pi_1}T+\frac{0}{\pi_0}T^p,\text{ et }\widetilde{e}(T)=1+\frac{1}{\pi_1}T+\frac{1}{2!\cdot(\pi_1)^2}T^2+\ldots+\frac{1}{p!\cdot(\pi_1)^p}T^p.$$ En examinant les conditions d’intégralité de chacun des termes de~$\tilde{e}(T)$, nous trouvons, pour les~$p-1$ premiers termes non constants, une condition de la forme~$\abs{T}^i\leq\abs{\pi_1}^i$, condition équivalente à
\begin{equation}\label{conditionpi2}\abs{T}\leq\abs{\pi_1},\end{equation}
la condition la plus contraignante étant celle du dernier terme, celui d’ordre~$p$, laquelle s’écrit
\begin{equation}\abs{T}\leq\abs{p}^{1/p}\cdot\abs{\pi_1}.\end{equation}
On retrouve bien une condition équivalente à~\eqref{conditionexp} en vertu de~\eqref{rayon01}.
\end{example}

\begin{example}[\emph{La série~$\exp(T+T^p/p)$}] Prenons~$P(T)=T+T^p/p$ et~$d=1$. Alors nous obtenons 
 $\widetilde{P}(T)=\frac{1}{\pi_1}T+\frac{1}{\pi_0}T^p,$ et calculons
\begin{equation}\label{etilde3}\widetilde{e}(T)=\left[1+\frac{1}{\pi_1}T+\frac{1}{2!\cdot(\pi_1)^2}T^2+\ldots+\frac{1}{p!\cdot(\pi_1)^p}T^p\right]+\left[\frac{1}{p\cdot(\pi_0)}T^p\right].\end{equation}
Comme précédemment, les~$p-1$ premiers termes non constants redonnent la condition~\eqref{conditionpi2}. La condition d’ordre~$p$ demande un peu de travail. Simplifions 
$$\left[\frac{1}{p!\cdot(\pi_1)^p}T^p\right]+\left[\frac{1}{p\cdot(\pi_0)}T^p\right]=\frac{(\pi_0)+(p-1)!\cdot(\pi_1)^p}{p\cdot(\pi_1)^p\cdot\pi_0}T^p.$$

Examinons le dénominateur, modulo~$p(\pi_1)^2$. Nous avons
$$(p-1)!(\pi_1)^p\equiv -(\pi_1)^p \pmod{p(\pi_1)^2}$$
d’où, rappelant~${\pi_1}^p+p\pi_1\equiv \pi_0 \pmod{p(\pi_1)^2}$,
$$(\pi_0)+(p-1)!\cdot(\pi_1)^p\equiv p\cdot\pi_1
\pmod{p(\pi_1)^2}.$$
Nous déduisons~$\abs{(\pi_0)+(p-1)!\cdot(\pi_1)^p}=\abs{p\pi_1}$. La condition d’intégralité du terme d’ordre~$p$
$$\abs{\frac{(\pi_0)+(p-1)!\cdot(\pi_1)^p}{p\cdot(\pi_1)^p\cdot\pi_0}T^p}\leq1$$
devient donc
$$\abs{\frac{p\pi_1}{p\cdot(\pi_1)^p\cdot\pi_0}T^p}=\abs{\frac{1}{(\pi_1)^{p-1}\cdot\pi_0}T^p}\leq1,\text{ soit }\abs{T}\leq\abs{(\pi_1)^{p-1}\cdot\pi_0}^{1/p}.$$
Finalement, nous calculons, d’après~\eqref{rayon0} et~\eqref{rayoni},
$$\abs{(\pi_1)^{p-1}\cdot\pi_0}=\abs{p}^{(p-1)/p(p-1)+1/(p-1)}=\abs{p}^{(2p-1)/p(p-1)}.$$
et la condition devient
\begin{equation}\label{conditionpi3}\abs{T}\leq\abs{p}^{\left.(2p-1)\middle/p^2(p-1)\right.}=\abs{\pi_2}^{2p-1}\end{equation}
ce qui corrobore~\cite[Tables p.427]{Robert}. Pour finir, vérifions que la condition~\eqref{conditionpi3} est bien plus contraignante que~\eqref{conditionpi2}. Nous avons
$$\abs{p}^{\left.(2p-1)\middle/p^2(p-1)\right.}=\abs{\pi_2}^{2p-1}=\abs{\pi_1}^2\abs{\pi_2}^{-1}=\abs{\pi_1}\cdot *$$
avec comme facteur de comparaison~$*=\abs{\pi_1/\pi_2}$ qui est bien~$<1.$
\end{example}

\begin{remark} Avec le changement de variable~$T=\pi U$, où~$\pi^{p-1}=-p$, nous en déduisons le rayon de la série
\begin{equation}\label{thetaDwork}
\theta(U)=\exp(X+X^p/p)=\exp(\pi(U-U^p)),
\end{equation}
qui donne la fonction de scindage de Dwork. À savoir
$$\abs{\pi U}\leq \abs{p}^{\left.(2p-1)\middle/p^2(p-1)\right.}$$
ou, remarquant~$\abs{\pi}=\abs{\pi_0}=\abs{p}^{1/(p-1)},$
$$\abs{ U}\leq 
\abs{p}^{\left.(2p-1-p^2)\middle/p^2(p-1)\right.}
=
\abs{p}^{-\left.(p-1)^2\middle/p^2(p-1)\right.}
=
\abs{p}^{-\left.(p-1)\middle/p^2\right.}.$$
Le signe, négatif, de l’exposant atteste de la «~surconvegence~» de la série~\eqref{thetaDwork}, c’est-à-dire du fait que son rayon de convergence majore strictement~$1$.
\end{remark}

\begin{example}[\emph{Les séries~$\exp(T+a_2T^2+a_3T^3+\ldots+a_{p-1}T^{p-1}+T^p/p)$}] Dans ce cas nous pouvons décomposer le problème selon les parties $p$-typiques:
$$\exp(T+T^p/p)\text{ de rayon }\abs{p}^{(2p-1)/p^2(p-1)}$$
et, pour~$2\leq i \leq p-1$,
$$\exp(a_iT^i)\text{ de rayon }\frac{\abs{p}^{1/(p-1)}}{\abs{a_i}^{1/i}}.$$
Le rayon de convergence de la série~$\exp(T+a_2T^2+a_3T^3+\ldots+a_{p-1}T^{p-1}+T^p/p)$ est
$$\min\left(\left\{\abs{p}^{(2p-1)/p^2(p-1)}\right\}\cup\left\{\frac{\abs{p}^{1/(p-1)}}{\abs{a_i}^{1/i}}~\middle|~2\leq i \leq p-1 \right\}\right),$$
même si le minimum est atteint plusieurs fois.

À titre d’exemple, si tous les coefficients~$a_i$ sont des entiers:~$\abs{a_i}\leq 1$, et que l’un au moins est une unité:~$\abs{a_i}=1$, alors le rayon de convergence est
$$\abs{p}^{1/(p-1)}.$$
\end{example}

\section{Sur l’Équivalence solubilité/integralité}\label{annexe} \emph{Nous nous plaçons dans le contexte du Théorème~\ref{thm2}. } 
Dans cette annexe, nous redonnons une démonstration de l’équivalence entre les deux premiers points du Théorème~\ref{thm2}. C’est également prétexte, 
à illustrer l’utilisation des méthodes déployées dans cet article, mais aussi à donner des estimations plus fines sur les séries étudiées, généralisant avantageusement~Lemmes~\ref{exo1} et~\ref{exo2}. Ces estimations devraient
s’avérer utiles dans des perspectives algorithmiques, ou servir de modèle de référence pour des équations différentielles plus générales.

\subsection{Argumentation}
Démontrons l’équivalence entre les deux premiers points du Théorème~\ref{thm2}. Le sens réciproque~$(2)\Rightarrow(1)$ de l’équivalence est clair: une série à coefficient entiers converge sur le disque unité ouvert.\footnote{Par exemple en utilisant  la formule~\eqref{formuleroc} plus bas.} Concernant le sens direct~$(1)\Rightarrow(2)$, nous souhaitons montrer que la série~$\exp(P(T))$, supposée de rayon au moins~$1$, est à coefficients entiers.

Rappelons que le rayon de convergence~$\rho$ de la série
\begin{equation}\label{truc}
e(T):= \sum_{i\geq0}b_i \alpha^i T^i
\end{equation}
est donné par la formule de Cauchy
\begin{equation}\label{formuleroc}1/\rho=\limsup_{i\geq0}\sqrt[i]{\abs{b_i}}.\end{equation}
Considérons également la quantité
\begin{equation}\label{formuleiota}1/\iota=\sup_{i\geq0}\sqrt[i]{\abs{b_i}},\end{equation}
dont l’inverse~$\iota=\iota(\exp(P(T)))$ sera nommé \emph{rayon d’intégralité de~$\exp(P(T))$}. En fait la série~$\exp(P(\alpha T))$ est à coefficients~$\alpha^i\cdot b_i$ tous entiers si et seulement si~$\abs{\alpha}\leq\iota$ (elle n’est de rayon de convergence au moins~$1$ que si~$\abs{\alpha}\leq\rho$).

Il suit immédiatement de~\eqref{formuleroc} et~\eqref{formuleiota} l’inégalité
\begin{equation}\label{inegalitetriviale} \iota\leq\rho.\end{equation}
L’implication~$(1)\Rightarrow(2)$ du Théorème~\ref{thm2} revient, dans le cas de la série~$e(T)=\exp(P(T))$ à montrer:
\begin{equation}\label{propriete}\rho(e(T))=1\Rightarrow \iota(e(T))=1.\end{equation}
Pour~$\alpha$ choisi non nul dans une extension ultramétrique de~$K$, on vérifie directement sur les formules~\eqref{formuleroc} et~\eqref{formuleiota},  qu’un changement de variable~$T\mapsto\alpha T$ induit les propriétés d’homogénéité 
\begin{equation}\label{homogene}\rho(e(\alpha T))=\rho(e(T))/\abs{\alpha}\text{ et }\iota(e(\alpha T))=\iota(e(T))/\abs{\alpha}.\end{equation}
Par homogénéité (et car on peut choisir~$\alpha$ dans une extension qui réalise tout nombre réel comme valeur absolue), la propriété~\ref{propriete} vaut pour toutes les séries~$e(T)$ de la forme exponentielle de polynômes si et seulement si
\begin{equation}\label{proprieteforte}
\forall\abs{\alpha}\in\mathbf{R}_{>0},~ \rho(e(T))=\abs{\alpha}\Rightarrow\iota(e(T))=\abs{\alpha}.\end{equation}
vaut pour toutes ces mêmes séries. En effet~\eqref{proprieteforte} est manifestement plus fort que~\eqref{propriete}, et~\eqref{proprieteforte} se ramène à~\eqref{propriete} par homogénéité.

Compte tenu de l’inégalité~\eqref{inegalitetriviale}, la propriété~\eqref{proprieteforte} se ramène à l’inégalité manquante
\begin{equation}\label{inegalitedure} \iota\geq\rho.\end{equation}
Par homogénéité, on peut supposer~$\iota=1$. Tout revient donc à montrer
\begin{equation}\label{iimplicationdure} \iota=1\Rightarrow\rho\geq1.\end{equation}

Pour cela nous utilisons\footnote{Soulignons que l’équivalence~$(2)\Leftrightarrow(3)$ a été démontrée indépendamment de l’équivlence~$(1)\Leftrightarrow(2)$ que nous cherchons ici à établir.} l’équivalence~$(2)\Leftrightarrow(3)$ dans la variante Corollaire~\ref{thm2bis} du Théorème~\ref{thm2}. En notant
$$\tilde{e}(T):=\exp\left(\tilde{P}(T)\right)=1+\sum_{i=1}^D \widetilde{b_i}T^i\pmod{T^{D+1}}\text{ et }1/\tilde{\iota}:=\sup_{1\leq i\leq D}\sqrt[i]{\abs{\widetilde{b_i}}}$$
il en ressort
$$\iota\geq 1 \Leftrightarrow \tilde{\iota}\geq1.$$
Là encore les deux expression sont homogènes, et on obtient en définitive
$$\iota=\tilde{\iota}.$$
Soit~$\upsilon=(u_1,u_2,\ldots,u_D)\in\W_D(K)$ le vecteur de Witt universel tronqué correspondant à~$\tilde{e}(T)$. D’après le Lemme~\ref{Lemmecritere}, l’intégralité de~$\upsilon$ équivaut à celle de~$\tilde{e}(T)$. Ainsi 
$$\tilde{\iota}\geq 1 \Leftrightarrow \sup_{1\leq i\leq D}\sqrt[i]{\abs{{u_i}}}\geq1.$$
Remarquons (\cite[(13.46), (9.27) with d=r=1, (13.59-60), (9.22) (cf. (6.26))]{Haze}) que~$\tilde{e}(\alpha T)$ correspond au produit~$(u_1,u_2,\ldots,u_D)\hat{\times}(\alpha,0,\ldots,0)$, qui (\cite[cf.~(13.47)]{Haze}, \cite[AC~IX.11~(39)]{BBKAC8}) vaut en fait~$(u_1\cdot\alpha,u_2\cdot\alpha^2,\ldots,u_D\cdot\alpha^D)$. Nous pouvons encore raisonner par homogénéité, et nous en tirons
$$\tilde{\iota} = \sup_{1\leq i\leq D}\sqrt[i]{\abs{{u_i}}}.$$
Notre argumentation est terminée si nous montrons l’énoncé suivant.
\begin{prop} Avec les notations précédentes, pour~$e(T)$ de la forme~$\exp(P(T))$ comme dans le Théorème~\ref{thm2}, nous avons
$$\rho=\iota=\tilde{\iota}=\sup_{1\leq i\leq D}\sqrt[i]{\abs{{u_i}}}.$$
\end{prop}
La seule égalité non encore montrée est~$\rho=\iota$, où ce qui revient au même, d’après les autres égalités,~$\rho=\sup_{1\leq i\leq D}\sqrt[i]{\abs{{u_i}}}$. Nous avons vu qu’il suffit de montrer
$$\sup_{1\leq i\leq D}\sqrt[i]{\abs{{u_i}}}=1\Rightarrow \rho\geq 1.$$
C’est le contenu du Corollaire~\ref{Coroestim} qui suit.

\subsection{Quelques estimations}Notons~$\upsilon_D\in\W(R)$ le vecteur de Witt universel correspondant à la série
\begin{equation}\label{piexpunivannexe}
 \exp\left((\zeta_{d_1}-1)\frac{T^1}{1}+(\zeta_{d_2}-1)\frac{T^2}{2}+\ldots+(\zeta_{d_D}-1)\frac{T^D}{D} \right).\end{equation}
Dans la discussion précédente la série~$\widetilde{e}(T)$ était donnée par le vecteur de Witt tronqué~$v=(u_1,\ldots,u_D)$. La série~$e(T)$ est alors donnée par le vecteur de Witt non tronqué~$v\hat\times \upsilon_D$. (Le produit est bien défini.)

\begin{prop}\label{Propestim} Soit~$v=(u_1,\ldots,u_D)\in\W_D(R)$ tel que~$\sup_{1\leq i\leq D}\abs{u_i}=1$. Alors les coefficients de la série
\begin{equation}\label{sérieprop}
1+\sum_{i\geq1} a_i T^i\end{equation}
correspondant au vecteur de Witt~$v\hat\times \upsilon_D$ satisfont
\begin{equation}
\label{coeffasymptuniv}
\abs{\pi_0}\leq\limsup_{i\geq0}\abs{a_i}< 1
\end{equation}
\end{prop}
L’estimation~\eqref{coeffasymptuniv} implique, via la formule de Cauchy~\eqref{formuleroc}, l’énoncé suivant.
\begin{coro}\label{Coroestim} La série~\eqref{sérieprop} a rayon de convergence~$1$.
\end{coro}
\begin{proof}[Preuve de la Proposition~\ref{Propestim}] Nous faisons une réduction aux sorite détaillé à la section suivante.

Montrons tout d’abord la majoration stricte de~\eqref{coeffasymptuniv}. Tout d’abord la série~\eqref{piexpunivannexe} est congrue à~$1$ modulo~$\pi_d$ avec~$d={\floor{\log_p(D)}}$: cela peut se montrer sur chaque composantes~$p$-typique prise à part et on est alors ramené à la Proposition~\ref{propannexe} (cf.~\eqref{piexpannexe}). Le vecteur de Witt correspondant~$\upsilon_D$ est donc nul modulo~$\pi_{\floor{\log_p(D)}}$. Son produit~$v\widehat\times \upsilon_D$ avec le vecteur de Witt~$\upsilon_D$ entier sera encore nul modulo~$\pi_d$. Donc la série
correspondante sera congrue à~$1$ modulo~$\pi_d$. 
Donc~$\sup\{\abs{a_i}\,|\, i\geq1\}\leq\abs{\pi_d}<1$.

Montrons la minoration. 
Commençons par utiliser la décomposition en composante $p$-typiques~\ref{factorisationtronquee} de~$v\in\W_D(R)$ dans~$\prod_{m\geq 1, p\nmid m} W_{d_m}$, qui  permet de factoriser la série~\eqref{sérieprop}, disons~$f(T)$ en
\begin{equation}\label{unautrelabel}
f(T)=\prod_{m\geq1, p\nmid m} f_m(T^m),
\end{equation}
où les séries~$f_m(T^m)$ sont $p$-typiques.
Chaque~$f_m(T)$ correspond à un vecteur~$v_m\hat{\times}w_{d_m}$ (où $w_{d_m}$ provient du Théorème~\ref{thm1}). Nous avons~$v=\sum V_m(v_m)$ en termes des opérations~$V_m$ (cf. \eqref{factorisationU}).
Par hypothèse~$v$ est entier et sa réduction dans~$\W_D(R/\sqrt{(p)})$ est non nulle. Donc chaque~$v_m$ est entier et l’un au moins est non nul modulo~$\sqrt{(p)}$.

Si~$v_m$ est nul modulo~$\sqrt{(p)}$, alors la majoration ...

\subparagraph*{Cas $p$-typique:}
Commençons par le cas où~$v=V_m(v_m)$ pour un certain~$m$, voire~$v=v_m$ car~$f_m(T)$ et~$f_m(T^m)$ partagent les même coefficients non nuls, bien que leur indices diffèrent. Abrégeons~$w=v_m$ et~$e=d_m$. Écrivons~$w=(x_0,\ldots,x_e)$ un tel vecteur~$p$-typique et décomposons-le
$$(x_0,\ldots,x_e)=(x_0,0,\ldots,0)~\widehat{+}~\ldots~\widehat{+}~(0,\ldots,0,x_e).$$
Par distributivité~$w\widehat{\times}w_e=(x_0,0,\ldots,0)\widehat{\times}w_e~\widehat{+}~\ldots~\widehat{+}~(0,\ldots,0,x_e)\widehat{\times}w_e.$
Notons
\begin{equation}\label{unlabel}
g_1(T)\cdot\ldots\cdot g_e(T)
\end{equation}
la factorisation correspondante. Nous avons~$(0,\ldots,0,x_i,0,\ldots,0)=V^i(x_i,0,\ldots)$ d’où
$$(0,\ldots,0,x_i,0,\ldots,0)\widehat{\times}w_e=V^i\left((x_i,0,\ldots)\right)\widehat{\times}w_e=V^i\left((x_i,0,\ldots)\widehat{\times}F^iw_{e}\right).$$
Or~$F^iw_e=w_{e-i}$. Il suit que~$f_i(T)$ est donnée par~$e_{e-i}(x_i T^{p^i})$  avec la notation précédent la Proposition~\ref{propannexe}. 

Si~$x_i$ est une unité, alors la Proposition~\ref{propannexe} nous apprend que la suite des valeur absolue des coefficients de~$e_{e-i}(x_i T^{p^i})$ a borne supérieure~$\abs{\pi_{e-i}}$.  C’est donc un polynôme modulo~$(\pi_{e-j})$ pour~$i<j\leq e$. Il suit que si~$\abs{x_i}<1$, alors~$e_{e-i}(x_i T^{p^i})$ a rayon de convergence $1/\abs{x_i}>1$ et alors~$e_{e-i}(x_i T^{p^i})$ est un polynôme modulo~$(p)$. Soit~$i_0$ minimal tel que~$\abs{x_{i_0}}=1$. Alors, modulo~$(\pi_{e-i_0})\cdot\sqrt(p)$, les facteurs de~\eqref{unlabel} sont tous des polynômes, mis à part~$g_{i_0}(T)$.

\subparagraph*{Conclusion:} Revenant au cas général~\eqref{unautrelabel}, nous sommes dans le cas du Corollaire~\ref{corosorite} dans le contexte de la Proposition~\ref{Prop6sorite}. La conclusion du Corollaire~\ref{corosorite} nous permet de conclure cette preuve.
\end{proof}

\subsection{Sorite calculatoire}\label{Sorite}
Les deux lemmes suivant résultent de l’estimation classique de la valeur absolue $p$-adique des coefficients de la série~$\exp(T)$. Le premier lemme sert aussi de définition de la notation~\eqref{notationupsilond}.
\begin{lemme}\label{exo1} Pour tout entier~$d\geq0$, le coefficient du terme~$(\pi_0)^{p^d}T^{(p^d)}/p^d!$ de degré degré~$p^d$
de la série~$\exp(\pi_0T)$ est s’écrit~
\begin{equation}\label{notationupsilond}
\frac{{\pi_0}^{p^d} }{p^d!}=\pi_0\cdot \upsilon_d\end{equation}
où~$\upsilon_d$ est une unité dans~$\mathbf{Z}_p$.
\end{lemme}
\begin{lemme}\label{exo2} La série~$e(T):=\exp(\pi_0 T)$ vérifie les congruences
\begin{equation}\label{cong1} e(T)\equiv 1 \pmod{\pi_0},\end{equation}
\begin{equation}\label{cong2} e(T)\equiv 1 + \sum_{d} \pi_0\cdot \upsilon_d T^{(p^d)}\pmod{(\pi_0)^2}.
\end{equation}
\end{lemme}
Ces deux lemmes démontrent le cas~$d=0$ de la proposition ci-dessous.
\begin{prop} \label{propannexe}
Soit~$d\geq 0$ et considérons la série
\begin{equation}\label{piexpannexe}
1+\sum_{i\geq1} a_i T^i=\exp\left(\pi_dT+\pi_{d-1}T^p/p+\ldots+\pi_{0}T^{p^d}/p^d \right).
\end{equation}
Alors~$\abs{a_i}\leq\abs{\pi_d}$ pour tout~$i\geq1$, avec égalité si et seulement si~$i$ est une puissance de~$p$.
\end{prop}
\begin{proof} Nous raisonnons par récurrence sur~$d$. Le cas~$d=0$ a déjà été vu en~\eqref{cong2}. Nous pouvons ici supposer~$R$ métriquement complet, voire que son groupe de valuation est dense. Notons suggestivement~$(\pi_d)^{1+\eps}$ l’idéal produit de l’idéal principal~$(\pi_d)$ par l’idéal maximal de~$R$.

Prenons le cas~$d\geq 1$ et travaillons dans~$R/(\pi_{d-1})^{1+\eps}$. Nous avons~$p=0$ dans~$R/(\pi_{d-1})^{1+\eps}$ vu que~$d\geq1$. D’après~\cite[AC~IX.15, Prop.~5~(51)]{BBKAC8}, le Frobenius (au sens des vecteurs de Witt) agit coefficient par coefficient  sur~$W(R/(\pi_{d-1})^{1+\eps})$ par élévation à la puissance~$p$-ième (le Frobenius de~$R/(\pi_{d-1})^{1+\eps}$). Or nous avons
\begin{equation*} F(w_{d})=w_{d-1}.\end{equation*}
Par hypothèse de récurrence, la série~$\sum_{i\geq0} b_i T^i$ associée à~$w_{d-1}$ a des coefficients~$b_i$ tels que:~$\abs{b_i}\leq\abs{\pi_d}$ pour tout~$i\geq1$, avec égalité si et seulement si~$i$ est une puissance de~$p$.

La construction de la série associée à un vecteur de Witt sur~$R/(\pi_{d-1})^{1+\eps}$ est une opération algébrique sur~$R/(\pi_{d-1})^{1+\eps}$: c’est manifeste en termes des coordonnéées de Witt universelles
$$ (w_n)_{n\geq1}\mapsto \sum_{i\geq0}c_iT^i=\prod_{n\geq0}(1-w_nT^n).$$
Cette opération est équivariante pour l’endomorphisme~$x\mapsto x^p$, au sens de l’application:~$({w_n}^p)_{n\geq1}\pmod{p}\mapsto\sum_{i\geq0}{c_i}^pT^i\pmod{p}$. Comme~$p=0$ dans~$R/(\pi_{d-1})^{1+\eps}$, nous l’identité~$F(w_{d})=w_{d-1}$ donne~${a_i}^p=b_i$ dans~$R/(\pi_{d-1})^{1+\eps}$, indice par indice. Rappelons que nous avons la congruence~${\pi_d}^p\equiv \pi_{d-1}\pmod{p}$, d’où l’identité d’idéaux~$(\pi_{d-1})^{1+\eps}={((\pi_{d})^{1+\eps})}^p$. Donc l’inégalité~$\abs{a_i}\leq \abs{\pi_d}$ (resp.~$\abs{b_i}<\abs{\pi_{d-1}}$) équivaut à~$\abs{a_i}\leq \abs{\pi_d}$ (resp.~$\abs{b_i}<\abs{\pi_{d-1}}$).

\end{proof}
Remarquons que cet énoncé montre en particulier la congruence
\begin{equation}\label{piexpannexe}
\exp\left(\pi_dT+\pi_{d-1}T^p/p+\ldots+\pi_{0}T^{p^d}/p^d \right)\equiv 1 \pmod{\pi_d}
\end{equation}
et plus particulièrement encore l’intégralité de cette série.

Pour tout entier~$d\geq0$, notons~$e_d$ la série de~\eqref{piexpannexe}. Généralisons les estimations précédentes au cas de produits des séries étudiées.
\begin{prop}\label{Prop6sorite}
 Soient~$\left(u_m\right)_{m\geq1,\,p\nmid m}$ une famille dans~$R$ et~$\left(d_m\right)_{m\geq1,\,p\nmid m}$ une famille d’entiers, de borne supérieure~$d$. Alors la série
$$1+\sum_{i\geq1} b_i T^i=\prod_{m\geq 1,\,p\nmid m}e_{d_m}(u_m T^{m})$$
est telle que~$\abs{b_i}\leq\abs{\pi_d}$ pour tout~$i\geq1$, avec égalité si et seulement si~$i$ est de la forme~$mp^n$ avec~$p\nmid m$, avec~$\abs{u_m}=1$ et avec~$d_m=d$.
\end{prop}
\begin{proof}On se place dans~$R/(\pi_d)^2$. Réécrivons le produit grâce à la Proposition~\ref{propannexe},
 en utilisant le symbole~$\ast_{n,m}$ pour désigner des unités de~$R$ que l’on ne souhaite pas expliciter,
$$\prod_{i\geq 0,\,p\nmid m}e_{d_m}(u_m T^{m})
\equiv
\prod_{i\geq 0,\,p\nmid m}\left(1+\pi_{d_m}\sum_{n\geq0} \ast_{n,m} {u_mT^m}^{p^n}\right)\pmod{(\pi_{d})^2}.$$
Il suffit de développer le produit de droite. Les termes faisant apparaître au moins deux fois un facteur~$\pi_d$, ou une fois un facteur~$\pi_{d^\prime}$ avec~$d\leq d^\prime$, sont congrus à~$0$ modulo~$(\pi_d)^2$. Il ne reste que
$$1+\sum_{i\geq1} b_i T^i=1+\sum_{m\geq 0,~p\nmid m, d_m=d}\pi_d\sum_{n\geq0} \ast_{n,m} \left(u_mT^m\right)^{p^n} \pmod{(\pi_d)^2}.$$
Il s’ensuit une formule directe pour les coefficients~$b_i$ modulo~$(\pi_d)^2$. L’énoncé en découle immédiatement.
\end{proof}
\begin{coro}\label{corosorite}
 Soit~$P(T)$ un polynôme dans~$R[T]$ et soit~$\eps>0$ tel que~$P\equiv 1\pmod{(\pi_d)^\eps}$. Alors la série
$$1+\sum_{i\geq1} c_i T^i=P(T)\cdot\left(1+\sum_{i\geq1} b_i T^i\right)$$
est telle que~$\abs{c_i}\leq\abs{\pi_d}$ pour tout~$i>\deg(P)$, avec égalité, pour de tels indices~$i$, si et seulement si~$i$ est de la forme~$mp^n$ avec~$p\nmid m$, avec~$\abs{u_m}=1$ et avec~$d_m=d$.
\end{coro}
\begin{proof}On développe le produit et montre comme précédemment que
$$1+\sum_{i\geq1} c_i T^i=P(T)+\sum_{m\geq 0,~p\nmid m, d_m=d}\pi_d\sum_{n\geq0} \ast_{n,m} \left(u_mT^m\right)^{p^n} \pmod{(\pi_d)^{1+\eps}}.$$
\end{proof}
\begin{remark}[Application du Corollaire] Les estimations que nous venons d’établir peuvent permettre de donner un nouvel algorithme « naïf » de calcul de rayon de convergence. Soit~$K$ est une extension de~$\QQ_p$ d’indice de ramification~$e$ fini, et~$P(T)$ un polynôme de~$K[T]$ de degré~$D$ tel que~$P(0)=0$. Alors il est possible de donner, en termes de~$D$ et de~$e$ un ensemble fini~$F(e,D)$ d’indices tel que pour tout polynôme de degré au plus~$D$ de~$K[T]$ tel que~$P(0)=0$, le rayon de convergence de~$\exp(P(T))=\sum_{i\geq0}a_iT^i$ est déterminé directement par la famille des valeurs absolues~$(\abs{a_f})_{f\in F(e,D)}$.
\end{remark}

\section{Sur les paramètres de Pulita et les groupes de Lubin-Tate}
Jusqu’à présent, nous nous sommes placé dans un cadre un peu plus restreint que celui de~\cite{Pulita}. La théorie de Pulita permet de considérer des paramètres~$\pi_i$ obtenus en fixant une loi de groupe de Lubin-Tate isomorphe à la loi multiplicative (\cite[§2.1 (2.1) p.509, Définition 2.2 et Remarque~2.3 p.510]{Pulita}, cf.~\ref{parametresPulita}~\emph{infra}). Du point de vue de~\cite{Pulita}, nous nous sommes restreint au cas
\begin{equation}\label{casmult}\pi_i=\zeta_i-1\end{equation} (cf. le début de la section~\ref{sectionpreuve} et celui de l’annexe~\ref{exemples}), ce qui correspond à la loi multiplicative
\begin{equation}\label{loimult} F_m(X,Y)=(X+1)(Y+1)-1=XY+X+Y\end{equation}
et au polynôme de Lubin-Tate~$P_m(X)=(X+1)^p-1$.

Dans cet annexe, nous montrons comment généraliser tous nos résultats dans le contexte de~\cite{Pulita}. Commençons par rappeler ce contexte.
\subsection{Paramètres de Pulita}\label{parametresPulita} Suivant~\cite{Pulita}, soit~$F(X,Y)$ une loi de groupe formel de Lubin-Tate définie sur~$\QQ_p$, supposée isomorphe à la loi multiplicative~\eqref{loimult}. Il existe un polynôme~$P(T)$ de la forme~$pT\pmod{T^2}$ et~$T^p\pmod{p}$ qui définit la multiplication par~$p$ au sens de la loi~$F$.

Nous choisissons une suite~$(\pi^F_i)_{i\geq0}$ telle que~$P(\pi^F_{i+1})=\pi^F_i$, que~$P(\pi^F_0)=0$ et~$\pi_0\neq0$.

Par hypothèse il existe un isomorphisme de~$F_m$ vers~$F$. Il s’agit d’une série~$\varphi(X)\in T\QQ_p[[T]]$ telle que~$\varphi(F_m(X,Y))=F(\varphi(X),\varphi(Y))$. Si~$\exp_F(T)$ est la série exponentielle de la loi~$F$, nous avons~$\varphi=\exp_F\circ\log$. Notons que la série~$\varphi$ a nécessairement tous ses coefficients dans~$\mathbf{Z}
_p$.

Notons~$\psi$ l’isomorphisme réciproque de~$F$ vers~$F_m$. Nous avons~$\psi=\stackrel{-1}{\varphi}=\exp\circ\log_F$ où~$\log_F$ est le logarithme\footnote{La série~$\log_F(T)$ est donnée explicitement par~$\log_F(T)=\lim_{n\to\infty}\left({P_F}^{\circ n}(T)/p^n)\right)$, où l’exposant~$\circ n$ dénote~$n$ compositions successives:~${P_F}^{\circ n}(t):=P_F(P_F(\cdots P_F(T)\cdots))$.}
 de la loi~$F$.
 
 Alors les nombres~${\pi^m_i}:=\psi(\pi^F_i)$, pour~$i\geq0$, sont tels que~$\zeta_i:=1+{\pi^m_i}$ est une racine de l’unité d’ordre~$p^{1+i}$. La suite des~$\zeta_i$ vérifie également la propriété de compatibilité~$\left(\zeta_{i+1}\right)^p=\zeta_i$. Autrement dit, en posant~$\pi_i=\pi^m_i$, on se trouve dans le cas~\eqref{casmult}.

Le lemme suivant donne des approximations polynomiales de la série~$\psi$.
\begin{lemme}\label{lm} Soit~$\epsilon>0$. Alors il existe un polynôme~$\Phi(X)\in\mathbf{Z}[X]$, de terme constant non nul, tel que
\begin{equation}\label{eqnlm} \forall0\leq i\leq d, \abs{\Phi(\pi^m_i)-\pi^F_i}\leq\epsilon.\end{equation}
\end{lemme}
\begin{proof} Cherchons tout d’abord~$\Phi$ à coefficients dans~$\mathbf{Z}_p$. Pour cela considérons, pour tout entier~$M$, la troncation~$\phi^{[M]}$ de la série~$\phi$ jusqu’à l’ordre~$M$ exclu. Alors la queue~$\phi-\phi^{[M]}$ de la série est dans~$X^M\mathbf{Z}_p[X]$. Appliquons en l’entier ultramétrique~$\pi^M_i$, ce qui donne la majoration
$$\phi(\pi^m_i)-\phi^{[M]}(\pi^m_i)\leq \abs{\pi^m_i}^M=\abs{p}^{M/(p^i(p-1))}.$$
Nous retrouvons~\eqref{eqnlm} dès que~$M$ est suffisamment grand.
Fixons un tel~$M$.

Pour obtenir~$\Phi$ à coefficients entiers, il suffit de le construire en approchant chaque coefficient de~$\Phi^{[M]}$ à la précision~$\eps/2$ par un entier.

\end{proof}

\subsection{}
La méthode que nous avons employée et nos démonstrations s’adaptent sans changements, à l’exception ci-dessous près, en remplaçant~$\zeta_i-1$ par~$\pi_i$.

\begin{prop}\label{propintegralitePulita} Soit~$d\geq0$. La série
\begin{equation}\label{integralitePulita}\exp\left(\pi^F_dT+\pi^F_{d-1}T^p/p+\ldots+\pi^F_0T^{p^d}/p^d\right)\end{equation}
est à coefficients dans~$\mathbf{Z}_{(p)}[\pi^F_d,\pi^F_{d-1},\ldots,\pi^F_0]$.
\end{prop}
Dans le cas~\eqref{casmult}, cette proposition a été obtenue au numéro~\ref{numeromatsuda}, dans lequel nous nous sommes référé à l’argument immédiat de~\cite[Lemme~1.5]{Matsuda}, qui se base sur les propriétés d’intégralité de l’exponentielle d’Artin-Hasse.

Dans son contexte plus général,~\cite{Pulita} utilise un autre type d’argument, l’«~astuce~»~\cite[§2.1, cf. Lemma~2.1, Remark~2.1]{Pulita}. Son résultat~\cite[Lemme~2.1, comme dans Remarque~2.3]{Pulita} est même plus général que notre énoncé. Nous nous proposons dans cette annexe de retrouver la propriété d’intégralité de~\eqref{integralitePulita} à partir de celle du cas~\eqref{casmult} déjà couvert.

D’une part cela permet d’étendre les résultats de ce document au contexte plus général de~\cite{Pulita}. En outre, notre démonstration est constructive et permet de rendre explicite
 l’intégralité cherchée.

\begin{proof}[Démonstration de la Proposition~\ref{propintegralitePulita}]
La série~\eqref{integralitePulita} a manifestement tous ses coefficients dans~$\QQ(\pi^F_d,\pi^F_{d-1},\ldots,\pi^F_0)$. En outre
$$\QQ(\pi^F_d,\pi^F_{d-1},\ldots,\pi^F_0)\cap\mathbf{Z}_p[\pi^F_d,\pi^F_{d-1},\ldots,\pi^F_0]=\mathbf{Z}_{(p)}[\pi^F_d,\pi^F_{d-1},\ldots,\pi^F_0].$$
Il suffit donc de montrer que les coefficients de~\eqref{integralitePulita} sont dans~$\mathbf{Z}_p[\pi^F_d,\pi^F_{d-1},\ldots,\pi^F_0]$.

Notons que~$\mathbf{Z}_p[\pi^F_d,\pi^F_{d-1},\ldots,\pi^F_0]$ est l’anneau d’entier de l’extension ramifiée associée au groupe de Lubin-Tate~$F$. Cet anneau s’écrit encore~$\mathbf{Z}_p[\pi^F_d]$, et même~$\mathbf{Z}_p[\pi^m_d]$ car il ne dépends de la loi~$F$ qu’à isomorphisme près.

tout revient à montrer que la série~\eqref{integralitePulita} appartient à~$\Lambda(\mathbf{Z}_p[\pi_d^m])$. D’après ce qui précède, cet série appartient à~~$\Lambda(\mathbf{Q}_p[\pi_d^m])$. Ses composantes fantômes s’écrivent
$$(\pi^F_d,\ldots,\pi^F_0,0,0,\ldots).$$
D’après le Lemme~\ref{lm}, il existe un polynôme à coefficients entiers et terme constant nul~$\Phi$ tel que
\begin{equation}\label{condition}
\forall 0\leq i\leq d, \abs{\Phi(\pi^m_i)-\pi^F_i}\leq \abs{p}^*.\end{equation}
Soit~$\Phi(AH(v_d))$ l’évaluation de~$\Phi$ dans l’anneau~$\Lambda(\mathbf{Z}_p[\pi_d^m])$, appliqué à la série~$AH(v_d)$. Alors~$\Phi(AH(v_d))$ est un élément de~$\Lambda(\mathbf{Z}_p[\pi_d^m])$. En outre, la série quotient
$$
\frac{\exp\left(\pi^F_dT+\pi^F_{d-1}T^p/p+\ldots+\pi^F_0T^{p^d}/p^d\right)}
{\Phi(AH(v_d))}
$$
a composantes fantômes
\begin{equation}\label{eqn}
(\phi_0,\phi_1,\ldots):=(\pi^F_d,\ldots,\pi^F_0,0,0,\ldots)-(\Phi(\pi^m_d),\ldots,\Phi(\pi^m_0),\Phi(0),\Phi(0),\ldots) \end{equation}
Comme~$\Phi$ est sans terme constant, on obtient la formule~$\forall i>d, \phi_i=0)$.

La série correspondant à~\eqref{eqn} est
\begin{equation}\label{expo}
\exp(\phi_0T+\phi_1T^p/p+\ldots+\phi_dT^{p^d}/p^d).\end{equation}
D’après~\eqref{condition}, l’argument de l’exponentielle dans~\eqref{expo}  est borné, en norme de Gauß, par le rayon de convergence de l’exponentielle.  La série de rayon~\eqref{expo} a donc rayon de convergence au moins~$1$. Elle est donc à coefficients entiers (Théorème~\ref{thm2}). Finalement
$$\exp\left(\pi^F_dT+\pi^F_{d-1}T^p/p+\ldots+\pi^F_0T^{p^d}/p^d\right)
=\Phi(AH(v_d))\cdot\exp(\phi_0T+\phi_1T^p/p+\ldots+\phi_dT^{p^d}/p^d)$$
est produit de série à coefficients entiers, donc est elle-même à coefficients entiers.
\end{proof}

\frenchspacing
\bibliographystyle{alpha}
\bibliography{piexp1ArXiv}





\end{document}